\documentclass{amsart}
\usepackage{amsmath,amsthm}
\usepackage{amssymb,url}

\title{Homomorphism reductions on Polish groups}
\author{Konstantinos A. Beros}
\date{}
\address{Department of Mathematics, University of North Texas, General Academics Building 435, 1155 Union Circle, \#311430, Denton, TX 76203-5017}
\email{beros@unt.edu}

\theoremstyle{plain}

\newtheorem{theorem}{Theorem}[section]
\newtheorem{lemma}[theorem]{Lemma}
\newtheorem{corollary}[theorem]{Corollary}

\newtheorem*{theorem*}{Theorem}
\newtheorem*{corollary*}{Corollary}
\newtheorem*{proposition*}{Proposition}

\theoremstyle{definition}

\newtheorem{definition}[theorem]{Definition}

\newtheorem{example}[theorem]{Example}

\theoremstyle{remark}

\newtheorem*{remark*}{Remark}
\numberwithin{equation}{section}

\def\concat{{}^\smallfrown}
\def\upto{\upharpoonright}
\def\len#1{\lvert #1 \rvert}

\def\supp{{\rm supp}}

\def\id{{\rm id}}
\def\NN{{\mathbb N}}
\def\ZZ{{\mathbb Z}}
\def\QQ{{\mathbb Q}}

\def\BS{{\mathcal N}}

\def\leqg{\leq_{\rm g}}

\begin{document}



\begin{abstract}
In an earlier paper, we introduced the following pre-order on the subgroups of a given Polish group:  if $G$ is a Polish group and $H,L \subseteq G$ are subgroups, we say $H$ is {\em homomorphism reducible} to $L$ iff there is a continuous group homomorphism $\varphi : G \rightarrow G$ such that $H = \varphi^{-1} (L)$.  We previously showed that there is a $K_\sigma$ subgroup, $L$, of the countable power of any locally compact Polish group, $G$, such that every $K_\sigma$ subgroup of $G^\omega$ is homomorphism reducible to $L$.  In the present work, we show that this fails in the countable power of the group of increasing homeomorphisms of the unit interval.
\end{abstract}

\maketitle



\section{Introduction}\label{S1}

In the study of definable equivalence relations, the equivalence relations arising as coset equivalence relations on Polish groups have played a key role.  For instance, the equivalence relation $E_0$ on $2^\omega$, given by
\[
x E_0 y \iff x(n) = y(n), \mbox{ for all but finitely many } n
\]
may be regarded as the coset equivalence relation of the subgroup
\[
\mathsf{Fin} = \{ x \in 2^\omega : x(n) = 0 \mbox{, for all but finitely many } n\}.
\]
of the Polish group $2^\omega$, where the group operation is bitwise addition mod 2.  Indeed, many important equivalence relations arise in this way, e.g., $E_1$, $E_2$, $E_3$, and $\ell^\infty$.  

The most important means of comparing equivalence relations is the quasiorder of {\em Borel reducibility}.  That is, if $E$ and $F$ are equivalence relations on Polish spaces $X$ and $Y$, respectively, one says that $E$ is {\em Borel reducible} to $F$ iff there is a Borel map $\varphi : X \rightarrow Y$ such that, for all $x, y \in X$, 
\[
x E y \iff \varphi (x) F \varphi (y).
\]
The map $\varphi$ is called a {\em Borel reduction} of $E$ to $F$.
The present work and our earlier paper \cite{BEROS.universal} are motivated by the observation that, in many cases, Borel reductions between coset equivalence relations can be witnessed by continuous group homomorphisms between appropriate groups.  For instance, let
\[
H_1 = \{ y \in 2^{\omega \times \omega} : (\forall^\infty m) (\forall m) (y(m,n) = 0)\},
\]
and let $E_1$ be the coset relation of $H_1$.  It is well-known that $E_0$ is Borel reducible to $E_1$.  A Borel reduction witnessing this is $\varphi : 2^\omega \rightarrow 2^{\omega \times \omega}$, given by
\[
\varphi (x) (m,n) = \varphi (x) (m),
\]
for each $x \in \ZZ_2^\omega$ and $m,n \in \omega$.  Note that, in fact, $\mathsf{Fin} = \varphi^{-1} [H_1]$, i.e., $\varphi$ is a continuous reduction of the subgroup $\mathsf{Fin}$ to $H_1$.  Furthermore, the map $\varphi$ induces an injective homomorphism
\[
\tilde\varphi : 2^\omega / \mathsf{Fin} \rightarrow 2^{\omega \times \omega} / H_1
\]
such that the diagram 
\[
\begin{array}{ccc}
2^\omega & \xrightarrow{\ \varphi \ } & 2^{\omega \times \omega}\\
\hspace{1.4ex}\left\downarrow {\scriptstyle p} \rule[0ex]{0ex}{2.5ex} \right .& & \hspace{1.4ex}\left\downarrow { \scriptstyle q}\rule[0ex]{0ex}{2.5ex}\right .\\
2^\omega / \mathsf{Fin} & \xrightarrow{\;\ \tilde\varphi \ \;} & \ZZ_2^{\omega \times \omega} / H_1
\end{array}
\]
commutes.  Here $p$ and $q$ denote the appropriate factor maps.

Abstracting from such specific examples leads to the following definition, already introduced in our paper \cite{BEROS.universal}.

\begin{definition}
Let $G$ be a Polish group.  For subgroups $H$ and $L$ of $G$, we say that $H$ is {\em group homomorphism reducible} to $L$ iff there is a continuous group homomorphism $\varphi : G \rightarrow G$ such that $H = \varphi^{-1} [L]$.  As a shorthand, we write $H \leqg L$.
\end{definition}

As an alternative formulation, observe that $H \leqg L$ iff there is an embedding of $G/H$ into $G/L$ with a lifting which is a continuous endomorphism of $G$.  Viewed in this light, our definition is not without precedent.  For instance, Farah \cite{FARAH.liftings} considers maps between factors of Boolean algebras with liftings which are algebra homomorphisms.

The $K_\sigma$ equivalence relations play an important role in the general theory of equivalence relations, e.g., $E_0$, $E_1$ and $\ell^\infty$ are all $K_\sigma$.  (Recall that a relation is $K_\sigma$ if it is the countable union of compact sets).  In fact, much study has been done of such relations.  See for instance \cite{ROSENDAL.cofinal}.  In light of this, we directed much of our attention to the $K_\sigma$ subgroups of Polish groups in \cite{BEROS.universal}.  We showed that there are $\leqg$-maximal $K_\sigma$ subgroups of $G^\omega$.  We referred to such $\leqg$-maximal $K_\sigma$ subgroups as {\em universal $K_\sigma$ subgroups} of $G^\omega$.  The equivalence relation analog is that of a Borel-complete $K_\sigma$ equivalence relation.  Though the requirement of universality is much stronger.  For instance, $\ell^\infty$ is a Borel-complete $K_\sigma$ equivalence relation, but it is not universal in our sense.

As a counterpoint, we showed that there are no universal $K_\sigma$ subgroups of $S_\infty$.  That is, for any $K_\sigma$ subgroup, $H$, of $S_\infty$, there is a $K_\sigma$ subgroup, $L \subseteq S_\infty$, such that $L \neq \varphi^{-1} [H]$, for every continuous homomorphism $\varphi : S_\infty \rightarrow S_\infty$.  Since $S_\infty$ continuously embeds its own countable power, it followed immediately that $S_\infty^\omega$ has no universal $K_\sigma$ subgroup.

The cases in which universal $K_\sigma$ subgroups do not exist are interesting in that they indicate a degree of rigidity of the group in question.  That is, if $G$ does not have a universal $K_\sigma$ subgroup, there is, in some sense, a shortage of continuous endomorphisms of $G$.  It is worth remarking upon the fact that $S_\infty$ does have $\leqg$-maximal analytic subgroups.  Thus, the rigidity which precludes universal $K_\sigma$ subgroups of $S_\infty$ is not sufficient to rule out universal analytic subgroups.

In the vein of our earlier results, we prove here an analogous theorem for the transformation group, $H_+([0,1])$, of increasing homeomporphisms of the unit interval.

\begin{theorem}\label{T1}
There is no universal $K_\sigma$ subgroup of $H_+([0,1])$.
\end{theorem}

The proof of this result is in Section~\ref{S3}.  As an aside, it follows from our earlier paper \cite[Theorem~1.7]{BEROS.universal} that there is an $F_\sigma$ subgroup of $H_+([0,1])$ of which every $K_\sigma$ subgroup is a continuous homomorphic pre-image.  By Theorem~\ref{T1}, such a ``universal $F_\sigma$ subgroup for $K_\sigma$'' cannot be $K_\sigma$ itself.

A direct corollary of Theorem~\ref{T1} is

\begin{corollary}
There is no universal $K_\sigma$ subgroup of $H_+([0,1])^\NN$.
\end{corollary}

\begin{proof}
Assume Theorem~\ref{T1} holds.  Observe that $H_+([0,1])^\NN$ is isomorphic to the closed subgroup 
\[
C = \{ f \in H_+([0,1]) : (\forall n) (f(1/n) = 1/n)\}
\]
of $H_+([0,1])$.  Let $\varphi : H_+([0,1])^\NN \rightarrow C$ be an isomorphism and let $\psi$ be an isomorphism between $H_+([0,1])$ and a closed subgroup of $H_+([0,1])^\NN$.  Suppose, towards a contradiction, that $H_+([0,1])^\NN$ has a universal $K_\sigma$ subgroup $H$.  Let $\tilde H$ be the $K_\sigma$ sugbroup $\varphi (H)$ of $H_+([0,1])$.  Fix an arbitrary $K_\sigma$ subgroup $L$ of $H_+([0,1])$.  Note that $\psi (L)$ is a $K_\sigma$ subgroup of $H_+([0,1])^\NN$ and hence, by the universality of $H$, there is a continuous endomorphism $\rho$ of $H_+([0,1])^\NN$ such that $\psi (L) = \rho^{-1} (H)$.  This implies that $L = (\varphi \circ \rho \circ \psi)^{-1} (H)$.  As $L$ was arbitrary, this shows that $\varphi (H)$ is a universal $K_\sigma$ subgroup of $H_+([0,1])$, contradicting Theorem~\ref{T1}.
\end{proof}

In our study of $K_\sigma$ subgroups from \cite{BEROS.universal}, we also considered situations in which every $K_\sigma$ subgroup of a given Polish group is compactly generated.  This is true in $\ZZ^\omega$, fo instance.  In the final section of this paper, we expand upon these ideas and describe a large class of groups in which every $K_\sigma$ is group homomorphism reducible to a compactly generated subgroup, i.e., the compactly generated subgroups are $\leqg$-cofinal in the $K_\sigma$ subgroups.



\section{Preliminaries and notation}\label{S2}

The group $H_+([0,1])$ of increasing homeomorphisms of $[0,1]$ is a Polish group when equipped with the topology of uniform convergence, i.e., the relative topology inherited from $C([0,1])$.  It is a peculiarity of $H_+([0,1])$ that this topology coincides with the topology of pointwise convergence.  In the present setting, the latter is much easier to manipulate.  The basic open sets in $H_+([0,1])$ are thus of the form
\[
U = \{ f \in H_+([0,1]) : (\forall i \leq k) (f (r_i) \in I_i)\},
\]
where $r_0 , \ldots , r_k \in [0,1]$ and $I_0 , \ldots , I_k \subseteq [0,1]$ are open intervals.

For $f \in H_+([0,1])$, let $\supp (f) = \{ x \in [0,1] : f(x) \neq x\}$ be the {\em support} of $f$.

Let $\BS$ denote the Baire space, i.e., $\NN^\NN$ with the product topology.  For $\alpha \in \BS$ and a finite sequence $s$ of natural numbers, $\alpha \supset s$ indicates that $s$ is an initial segment of $\alpha$.

\begin{definition}
Given $f \in H_+([0,1])$ and $\alpha \in \BS$, define $f$ to be {\em $\alpha$-continuous}, iff, for each $k \in \NN$ with $k \geq 1$, if $x,y \in [0,1]$ are such that $\lvert x - y \rvert < 1 / \alpha (k)$, then $\lvert f(x) - f(y) \rvert < 1/k$.
\end{definition}

Recall that a family $F \subseteq H_+ ([0,1])$ is {\em equicontinuous} iff there is an $\alpha \in \BS$ such that each $f \in F$ is $\alpha$-continuous.  For each $\alpha \in \BS$, let
\[
K_\alpha = \{ f \in H_+([0,1]) : f,f^{-1} \mbox{ are $\alpha$-continuous}\}
\]
and note that, by definition, each $K_\alpha$ is equicontinuous.  

\begin{lemma}
Suppose $f_0 , f_1 , \ldots \in H_+([0,1])$ and each $f_n$ is $\alpha$-continuous.  If $f_n \rightarrow f$ pointwise, then $f$ is also $\alpha$-continuous.
\end{lemma}

\begin{proof}
Fix $k \in \NN$ and $x,y \in [0,1]$ such that $\lvert x - y \rvert < 1/\alpha (k)$.  Assume that $k , \alpha (k) >1$, as otherwise the desired conclusion is automatic.  Assume that $x < y$ and, since $1/\alpha(k) \leq 1/2$, we may assume that $x \neq 0$ or $y \neq 1$.  These two cases are analogous, so assume $y \neq 1$.  Let $y' > y$ be such that $\lvert x - y' \rvert < 1/\alpha (k)$.  For each $n$, $\lvert f_n (x) - f_n (y') \rvert < 1/k$ and hence 
\[
\lvert f(x) - f(y') \rvert = \lim_n \lvert f_n (x) - f_n (y')\rvert \leq 1/k.
\]
Since $f$ is a strictly increasing function, 
\[
\lvert f(x) - f(y) \rvert < \lvert f(x) - f(y') \rvert \leq 1/k.
\]
This completes the proof.
\end{proof}

Each $K_\alpha$ is thus closed and, since the $K_\alpha$ are also equicontinuous, it follows from the Arzel\`a-Ascoli Theorem (see \cite[\S 7.10]{ROYDEN.real.analysis}) that each $K_\alpha$ is compact.  Conversely,

\begin{lemma}
If $K \subseteq H_+([0,1])$ is compact, then there exists $\alpha \in \BS$ such that each $f \in K$ is $\alpha$-continuous.
\end{lemma}

\begin{proof}
Fix $k \in \NN$, it suffices to determine a value for $\alpha (k)$ which is sufficiently large that $\lvert f(x) - f(y) \rvert < 1/k$, whenever $f \in K$ and $\lvert x-y \rvert < 1/ \alpha (k)$.  Indeed, fix open intervals $I_0 , \ldots , I_n \subseteq (0,1)$ which cover $(0,1)$ and such that, if each $I_i = (a_i , b_i)$, then 
\[
a_i < a_{i+1} < b_i, 
\]
for each $i \leq n-1$.  Moreover, choose the $I_i$ such that $\lvert b_{i+2} - a_i\rvert < 1/k$, for each $i \leq n-2$.  For each increasing tuple $r = \langle r(0) , r(1) , \ldots , r(n) \rangle \in (0,1)^{n+1}$, define the open set
\[
U_r = \{ f \in H_+([0,1]) : (\forall i \leq n) (f(r(i)) \in I_i\}
\]
and observe that the $U_r$ cover $H_+([0,1])$.  Since $K$ is compact, there exist $r_0 < r_1 < \ldots < r_p$ such that $U_{r_0} , \ldots , U_{r_p}$ cover $K$.  Let 
\[
\varepsilon_k = \min ( \{ \lvert r_j (i) - r_j (i+1) \rvert : i < n \ \& \ j \leq p\})
\]
and note that, for each $f \in K$ and $x,y \in [0,1]$, if $\lvert x - y \rvert < \varepsilon_k$, then $\lvert f(x) - f(y) \rvert < 1/k$.  Choose $\alpha (k) \geq 1/\varepsilon_k$.  This completes the proof.
\end{proof}

If $K \subseteq H_+([0,1])$ is compact, then so is $\{ f^{-1} : f \in K\}$.  It thus follows from the last lemma that $K \subseteq K_\alpha$, for some $\alpha \in \BS$.

In the interest of clarity, if $f,g$ are functions, $fg$ througout denotes the composite function $f \circ g$ and $f^n$ dentotes the $n$-fold iterate of $f$ (for $n \in \NN$).  Regarding elements of $\BS$ as functions on $\NN$, this notation applies equally to members of the Baire space.  Similarly, if $f \in H_+([0,1])$ and $\varphi : H_+([0,1]) \rightarrow H_+([0,1])$ is an endomorphism, $\varphi f$ denotes the image of $f$ under $\varphi$, i.e., $\varphi (f)$.

The following facts are direct consequences of the definitions of $\alpha$-continuity and the $K_\alpha$.
\begin{enumerate}
\item If $f$ is $\alpha$-continuous and $g$ is $\beta$-continuous, then $f g$ is $\beta \alpha$-continuous.
\item If $f \in K_\alpha$ and $g \in K_\beta$, then $fg \in K_{\max(\alpha \beta , \beta \alpha)}$, where $\max (\alpha \beta , \beta \alpha) \in \BS$ is such that $\max (\alpha \beta , \beta \alpha) (k) = \max (\alpha \beta (k) , \beta \alpha (k))$, for each $k \in \NN$.
\item If $f \in K_\alpha$ and $\lvert x-y \rvert \geq 1/k$, for some $k$, then $\lvert f(x) - f(y)\rvert \geq 1/\alpha (k)$.
\end{enumerate}

Finally, for $A \subseteq H_+([0,1])$, let $\langle A \rangle$ designate the subgroup generated by $A$.



\section{Proof of Theorem~\ref{T1}}\label{S3}

\begin{lemma}
There is an increasing homeomorphism $\sigma : [0,1] \rightarrow [0,1]$ such that $\sigma$ is supported on $[1/2 , 3/4]$ and the conjugacy class of $\sigma$ is dense.
\end{lemma}

\begin{proof}
Glasner-Weiss \cite{GLASNER-WEISS.topological.rokhlin} have shown that $H_+([0,1])$ contains a dense conjugacy class.  Therefore, fix $\sigma \in H_+([0,1])$ such that the conjugacy class of $\sigma$ is dense.  The objective of the proof is to modify $\sigma$ and produce $\tau$ which is supported on $[1/2,3/4]$ and still has dense conjugacy class.

Let $h : [0,1] \rightarrow [1/2 , 3/4]$ be a linear bijection.  Define $\tau \in H_+([0,1])$ by 
\[
\tau(x) = \begin{cases}
h \sigma h^{-1} (x) & \mbox{if } x \in [1/2 , 3/4],\\
x & \mbox{otherwise}.
\end{cases}
\]
It remains to show that the conjugacy class of $\tau$ is dense in $H_+([0,1])$.  To this end, fix a basic open set $U \subseteq H_+([0,1])$ with $0 < r_0 < \ldots < r_k < 1$ and nonempty open intervals $I_0 , \ldots , I_k \subseteq (0,1)$ such that 
\[
U = \{ f \in H_+([0,1]) : (\forall i \leq k) (f(r_i) \in I_i)\}.
\]
With no loss of generality, $\sup ( I_i )< 1$ and $\inf (I_i) > 0$, for each $i \leq k$, as this only shrinks the open set $U$.  Fix $a,b \in (0,1)$ such that $0 < a < r_0$, $r_k < b < 1$, $a < \min_{i \leq k} ( \inf (I_i))$ and $b > \max_{i \leq k} (\sup (I_i))$.  Let $g$ be the piecewise linear map defined such that the graph of $g$ has vertices $( 0,0 ) , ( 1/2 , a ) , ( 3/4 , b ) ,  ( 1,1 )$.  Note that $g \in H_+([0,1])$.

For each $i \leq k$, let $s_i = h^{-1} g^{-1} (r_i)$ and $J_i = h^{-1} g^{-1} (I_i)$.  Since the conjugacy class of $\sigma$ is dense, there exists $f \in H_+([0,1])$ such that $f \sigma f^{-1} (s_i) \in J_i$, for each $i \leq k$.  Define $f_1 \in H_+([0,1])$ by 
\[
f_1 (x) = \begin{cases}
hfh^{-1} (x) & \mbox{if } x \in [1/2 , 3/4],\\
x & \mbox{otherwise}.
\end{cases}
\]
To complete the proof, it suffices to show that $g f_1 \tau f_1^{-1}g^{-1} \in U$.  Fix $i \leq k$ and observe that 
\begin{align*}
gf_1 \tau f_1^{-1}g^{-1} (r_i)
&= g f_1 \tau f_1^{-1} h(s_i)\\
&= g h f h^{-1} h \sigma h^{-1} h f^{-1} h^{-1} h (s_i)\\
&= g h f \sigma f^{-1} (s_i)\\
&\in g h (J_i)\\
&= I_i.
\end{align*}
This completes the proof.
\end{proof}

Let $\sigma_1$ be as in the previous lemma, with dense conjugacy class and $\supp (\sigma_1) \subseteq [1/2 , 3/4]$.  Define $h \in H_+([0,1])$ such that the graph of $h$ is polygonal and the vertices of $h$ are the elements of the set
\[
\{ ( 1- 2^{-n} , 1 - 2^{-n-1} ) : n \in \NN, n \geq 1\}
\]
of points in $[0,1]^2$, i.e., $h$ has vertices $( 1/2 , 3/4 ) , ( 3/4 , 7/8 ) , \ldots$.  To simplify notation, let $I_n$ denote the interval $[1- 2^{-n} , 1 - 2^{-n-1}]$, for each $n \in \NN$.  In other words, $I_0 = [0,1/2]$, $I_1 = [1/2 , 3/4]$, and so forth.  For $n\geq 1$, define $\sigma_n = h^{n-1} \, \sigma_1 \, h^{1-n}$ and observe that $\sigma_n$ is supported on the interval $I_n$.

Given $\alpha \in \BS$, define $L_\alpha \subseteq H_+([0,1])$ to be the set of those $f \in H_+([0,1])$ such that $f \upto I_0 = \id \upto I_0$ and, for each $n \geq 1$, $f \upto I_n = \sigma_n^i \upto I_n$, for some $i \leq \alpha (n)$.  If $f \in L_\alpha$ is the (unique) member of $L_\alpha$ such that $f \upto I_n = \sigma_n^{\delta(n)} \upto I_n$, for a given $\delta \leq \alpha$, then denote $f$ by $f_\delta$.  Observe that $L_\alpha$ is a continuous image of the compact space 
\[
X_\alpha = \prod_n \{ 0, \ldots , \alpha(n)\}, 
\]
via the map $\delta \mapsto f_\delta$.  Each $L_\alpha$ is therefore compact.

In fact, the map $\delta \mapsto f_\delta$ is a homeomorphism between $X_\alpha$ and $L_\alpha$ with its subspace topology.  It follows that the sets 
\[
U_s = \{ f_\delta : \delta \supset s\},
\]
with $s$ a finite sequence bounded by $\alpha$, form a basis for the relative topology on $L_\alpha$.

Given a $K_\sigma$ subgroup $H \subseteq H_+([0,1])$, the objective of the proof is to show that $H$ is not a universal subgroup.  For this it is sufficient to show that there exists a compact set $L \subseteq H_+([0,1])$ such that, for every non-trivial continuous endomorphism $\varphi$ of $H_+([0,1])$, the image of $\langle L \rangle$ under $\varphi$ is not contained in $H$ and hence $\varphi$ is not a homomorphism reduction of $\langle L \rangle$ to $H$.  Note that restriction to non-trivial endomorphisms of $H_+([0,1])$ is appropriate since, if $\langle L \rangle = \varphi^{-1} (H)$, then $\ker (\varphi) \subseteq \langle L \rangle$ and hence $\ker (\varphi) \neq H_+([0,1])$, as otherwise $H_+([0,1]) = \langle L \rangle$, which is impossible because $H_+([0,1])$ is not itself $K_\sigma$.  

To further simplify the argument, assume that $H$ is of the form $\bigcup_m K_{\beta_m}$, where $K_{\beta_m}$ is defined as in Section~\ref{S2}.  It is possible to make this assumption as $H$ is contained in a set of the form $\bigcup_m K_{\beta_m}$.  In addition, assume that $\{ \beta_m : m \in \NN\}$ is closed under composition and taking pointwise maxima, i.e., for each $m,n$ there exist $p,q$ such that $\beta_m \beta_n = \beta_p$ and $\max (\beta_m , \beta_n) = \beta_q$.  The only purpose of these assumptions is to guarantee that, for any $m,n$, there exists $p$ such that if $f \in K_{\beta_m}$ and $g \in K_{\beta_n}$, then $fg , gf \in K_{\beta_p}$.

Observe that these simplifying assumptions introduce no loss of generality, since they only enlarge the set $\bigcup_m K_{\beta_m}$.

Henceforth, fix $\beta_0 , \beta_1 , \ldots \in \BS$ as above.  Choose $\alpha \in \BS$ such that, for each $n \in \NN$,
\[
\alpha (n) = \max ( \{ \beta_i^j (k) : i,j,k \leq n\}) + 1.
\]
To complete the proof, it suffices to show that there is no continuous endomorphism of $H_+([0,1])$ which maps the compactly generated subgroup $\langle \{ h \} \cup L_\alpha \rangle$ into $\bigcup_m K_{\beta_m}$.  (Here $h$ is the map described above such that $\sigma_{n+1} = h \sigma_n h^{-1}$.)  Indeed, suppose that, on the contrary, there exists an endomorphism $\varphi$ of $H_+([0,1])$ such that $\varphi$ maps $\langle \{ h \} \cup L_\alpha \rangle$ into $\bigcup_m K_{\beta_m}$.

\begin{lemma}
There exists $m$ such that $\varphi (\{ h \} \cup L_\alpha) \subseteq K_{\beta_m}$.
\end{lemma}

\begin{proof}
As a compact subset of a Polish space, $\varphi (L_\alpha)$ is Polish in its relative topology and thus, by the Baire Category Theorem, there exists an open set $V \subseteq H_+([0,1])$ and $m_0 \in \NN$ such that $\varphi (L_\alpha) \cap V$ is nonempty and $K_{\beta_{m_0}}$ is comeager on $\varphi (L_\alpha) \cap V$.  Since $K_{\beta_{m_0}}$ is a closed set, $K_{\beta_{m_0}}$ must, in fact, contain $\varphi (L_\alpha) \cap V$.  Let $U = \varphi^{-1} (V)$.  It follows that $U \cap L_\alpha$ is nonempty and hence $\varphi (U_s) \subseteq K_{\beta_{m_0}}$, for some finite sequence $s$ which is bitwise bounded by $\alpha$.  Observe that any $f \in L_\alpha$ is of the form 
\begin{equation}\label{E1}
\sigma_1^{i_1} \, \sigma_2^{i_2} \ldots \sigma_{\len s}^{i_{\len s}} g,
\end{equation}
for some $g \in U_s$ and $i_1 , \ldots , i_{\len s} \in \ZZ$, with each $\lvert i_p \rvert \leq 2\alpha(p)$.  Note that
\[
\sigma_1^{i_1}, \sigma_2^{i_2}, \ldots \sigma_{\len s}^{i_{\len s}} \in \langle L_\alpha \rangle,
\]
for all such $i_1 , \ldots , i_{\len s}$.  Since $\varphi (\langle L_\alpha \rangle) \subseteq \bigcup_m K_{\beta_m}$, it now follows from the properties of $\beta_0 , \beta_1 , \ldots$ that there exists $m_1 \in \NN$ such that $K_{\beta_{m_1}}$ contains all elements of the form \eqref{E1}.  In other words, $K_{\beta_{m_1}}$ contains $\varphi (L_\alpha)$.

Finally, since $\varphi h \in \bigcup_m K_{\beta_m}$, there exists $m_2$ with $\varphi h \in K_{\beta_{m_2}}$.  It follows from the closure properties of the $\beta_i$ that $\varphi (\{ h \} \cup L_\alpha) \subseteq K_{\beta_{m_3}}$, for some $m_3 \in \NN$.  This proves the lemma.
\end{proof}

Let $\beta_m$ be as in the previous lemma, with $\varphi (\{ h \} \cup L_\alpha) \subseteq K_{\beta_m}$.  For simplicity, write $\beta = \beta_m$. 

Since the conjugacy class of $\sigma_1$ is dense, if $\varphi \sigma_1 = \id$, then $\varphi f = \id$, for each $f \in H_+([0,1])$.  Indeed, if $\sigma_1 \in \ker (\varphi)$, then $\ker (\varphi)$ contains the entire conjugacy class of $\sigma_1$, as it is a normal subgroup of $H_+([0,1])$.  This implies that $\ker (\varphi) = H_+([0,1])$, since $\ker (\varphi)$ is also closed.  On the other hand, this violates the assumption that $\varphi$ is non-trivial.  Therefore, choose $x_1 \in [0,1]$ and $k \in \NN$ such that
\[
\lvert x_1 - \varphi \sigma_1 (x_1)\rvert \geq 1/k.
\]
Define $x_{n+1} = \varphi h (x_n)$, for each $n \geq 1$, where $h$ is again as above.  Observe that 
\begin{align*}
\lvert x_2 - \varphi \sigma_2 (x_2)\rvert
&= \lvert \varphi h (x_1) - \varphi \sigma_2 \, \varphi h (x_1)\rvert\\
&= \lvert \varphi h (x_1) - \varphi h \, \varphi \sigma_1 \, \varphi h^{-1} \, \varphi h (x_1)\rvert\\
&= \lvert \varphi h (x_1) - \varphi h \, \varphi \sigma_1 (x_1)\rvert\\
&\geq 1 / \beta(k),
\end{align*}
since $\varphi h \in K_\beta$ and hence $\varphi h^{-1}$ is $\beta$-continuous.  Iterating this argument, it follows that
\begin{equation}\label{E2}
\lvert x_n - \varphi \sigma_n (x_n)\rvert \geq 1 / \beta^{n-1} (k),
\end{equation}
for each $n \geq 1$.

Let $n \geq k,m$ and let $p = \alpha (n)$.  Recall that, by the choice of $\alpha$,
\[
p > \beta^n (k).
\]
Consider the points $x_n , \varphi \sigma_n (x_n) , \varphi \sigma_n^2 (x_n) , \ldots , \varphi \sigma_n^p (x_n)$, where $p \geq 1$.  Because $\varphi \sigma_n$ is an increasing homeomorphism of $[0,1]$ and $\varphi \sigma_n (x_n) \neq x_n$, either
\[
x_n < \varphi \sigma_n (x_n) < \varphi \sigma_n^2 (x_n) < \ldots < \varphi \sigma_n^p (x_n)
\]
or 
\[
x_n > \varphi \sigma_n (x_n) > \varphi \sigma_n^2 (x_n) > \ldots > \varphi \sigma_n^p (x_n).
\]
In either case, there must exist $i < p$ such that
\[
\lvert \varphi \sigma_n^i (x_n) - \varphi \sigma^{i+1}_n (x_n)\rvert \leq 1/p,
\]
since the length of $[0,1]$ is $1$.  Since $1/p < 1 / \beta^n (k)$, it follows that $\varphi \sigma_n^{-i}$ is not $\beta$-continuous, since
\[
\lvert \varphi \sigma_n^i (x_n) - \varphi \sigma^{i+1}_n (x_n) \rvert \leq 1/p < 1/\beta^n(k),
\]
but 
\[
\lvert \varphi \sigma_n^{-i} \, \varphi \sigma_n^i (x_n) - \varphi \sigma_n^{-i} \, \varphi \sigma_n^{i+1} (x_n)| = \lvert x_n - \varphi \sigma_n (x_n)\rvert \not < 1/\beta^{n-1}(k),
\]
by \eqref{E2}.  Thus $\varphi \sigma_n^i \notin K_\beta$.  This is a contradiction, since the choice of $\alpha$ implies that $\sigma_n^i \in L_\alpha$ and $\varphi$ maps $L_\alpha$ into $K_\beta$.
\qed



\section{Compactly generated subgroups}

We consider a criterion on a group $G$ which guarantees that the compactly generated subgroups of $G^\omega$ are $\leqg$-cofinal in the $K_\sigma$ subgroups.  We begin with a couple of motivating examples.

\begin{example}
We showed in \cite{BEROS.universal} that every $K_\sigma$ subgroup of $\ZZ^\omega$ is compactly generated.  The argument was similar in character to the proof that $\ZZ$ is a principal ideal domain.  For instance, the countable dense subgroup $\{ x \in \ZZ^\omega : (\forall^\infty n) (x(n) = 0)\}$ is generated by the compact set 
\[
\{ 0^n \concat 1 \concat \bar 0 :  n \in \omega\} \cup \{ \bar 0 \}.
\]
\end{example}

For the sake of the next example, we say that $x$ is {\em divisible in} $H$ (where $H$ is a group), iff, for each $n$, there exists $y \in H$ such that $y^n = x$.

\begin{example}
Consider the Polish group $\QQ^\omega$, where $\QQ$ is equipped with the discrete topology.  There are $K_\sigma$ subgroups, e.g.,
\[
H = \{ x \in \QQ^\omega : (\forall n) (n \geq 1 \implies x(n) = 0)\},
\]
which are $K_\sigma$, but not compactly generated.  In fact, the subgroup $H$ is not $\leqg$-reducible to any compactly generated subgroup of $\QQ^\omega$.  To see this, suppose that $\varphi : \QQ^\omega \rightarrow \QQ^\omega$ is a group homomorphism with $H = \varphi^{-1} (\langle K \rangle)$, for some compact set $K \subseteq \QQ^\omega$.  Note that, since $H$ is a divisible subgroup of $\QQ^\omega$, the subgroup $\langle K \rangle$ must contain elements which are divisible in $\langle K \rangle$.  This, however, is impossible since it would imply that there is some $n \in \omega$ such that $\{ x(n) : x \in K\}$ is infinite.
\end{example}

The results in this section were motivated by the observation that the last example hinges on the fact that there are no compactly generated subgroups, $H$, of $\QQ^\omega$ such that some element of $H$ is divisible in $H$.  On the other hand, for instance, in the permutation group of the natural numbers, there are compactly generated subgroups which have divisible elements.  Consider, for instance, the following.

\begin{example}
For each $n \geq 1$, let 
\[
\pi_n = \prod_{i \in \omega} [2ni , 2ni+2 , \ldots , 2n(i+1) , 2ni+1 , 2in+3 , \ldots , 2n(i+1)+1]
\]
and observe that $(\pi_n)^n = \pi_1$, for each $n \geq 1$.  It follows that $\pi_1$ is divisible in the subgroup generated by the $\pi_n$.  If we let $H$ be the subgroup generated by the $\pi_n$, together with the finite support permutations, we obtain a compactly generated subgroup with a divisible element.  Indeed, $H$ is compactly generated, since it is generated by the compact set
\[
\{ [n , n+1] : n \in \omega\} \cup \{ \pi_n' : n \in \omega\}\cup \{ \mathrm{id} \},
\]
where
\[
\pi_n' = \prod_{i \geq n} [2ni , 2ni+2 , \ldots , 2n(i+1) , 2ni+1 , 2in+3 , \ldots , 2n(i+1)+1].
\]
\end{example}

In turns out that, while there are non-compactly generated $K_\sigma$ subgroups of $S_\infty$, every $K_\sigma$ subgroup of $S_\infty$ is, in fact, $\leqg$-reducible to a compactly generated subgroup.  This is a consequence of Theorem~\ref{T2} below.

\begin{definition}
A subgroup $H$ of a topological group $G$ is {\em almost compactly generated} iff there is a a compact set $K \subseteq G$ and a continuous, injective group homomorphism $\varphi : G \rightarrow G$ such that $H = \varphi^{-1} (\langle K \rangle)$.
\end{definition}

The following theorem and remark together characterize those Polish groups of the form $G^\omega$ in which every $K_\sigma$ subgroup is almost compactly generated.

\begin{theorem}\label{T2}
Let $G$ be a Polish group.  If $G$ has a dense subgroup which is a one-to-one continuous homomorphic pre-image of some compactly generated subgroup of $G^\omega$, then every $K_\sigma$ subgroup of $G^\omega$ is almost compactly generated.
\end{theorem}

\begin{remark*}
Suppose that the hypothesis of Theorem~\ref{T2} fails, i.e., there is no dense subgroup of $G$ which is an injective homomorphic pre-image of some compactly generated subgroup of $G^\omega$.  One can define a $K_\sigma$ subgroup of $G^\omega$ which is not almost compactly generated by taking a dense $K_\sigma$ subgroup $D \subseteq G$ and letting
\[
\tilde D = \{ \xi \in G^\omega : (\forall i) (\xi (i) = \xi(0) \in D)\}.
\]
It follows that $\tilde D$ is not almost compactly generated.  Otherwise, there would exist an injective continuous endomorphism $\varphi : G^\omega \rightarrow G^\omega$ and a compact set $K \subseteq G^\omega$ such that $\tilde D = \varphi^{-1} (\langle K \rangle)$.  Let $\psi : G \rightarrow G^\omega$ be given by $\psi (x) (i) = x$, for all $i$.  Then, contrary to assumption, $D = (\varphi \circ \psi)^{-1} (\langle K \rangle)$.
\end{remark*}

Theorem~\ref{T2} yields the following immediate corollary.

\begin{corollary}\label{C1}
If $G$ is a Polish group with a dense compactly generated subgroup, then every $K_\sigma$ subgroup of $G^\omega$ is almost compactly generated.
\end{corollary}

In general, we can consider other groups besides those of the form $G^\omega$.

\begin{corollary}\label{C2}
Suppose that $G$ is a Polish group such that $G^\omega$ isomorphic to a subgroup of $G$.  If $G$ has a dense almost compactly generated subgroup, then every $K_\sigma$ subgroup of $G$ is almost compactly generated.
\end{corollary}

\begin{proof}
Let $\varphi_1 : G \rightarrow G^\omega$ and $\varphi_2 : G^\omega \rightarrow G$ be isomorphic embeddings.  Fix a $K_\sigma$ subgroup $H \subseteq G$.  By Theorem~\ref{T2}, there exists a compact set $K \subseteq G^\omega$ and a continuous injective homomorphism $\varphi : G^\omega \rightarrow G^\omega$ such that $\varphi_1 (H) = \varphi^{-1} (\langle K \rangle)$.  It now follows from the injectivity of $\varphi_2$ that 
\[
H = (\varphi_2 \circ \varphi \circ \varphi_1)^{-1} (\langle \varphi_2 (K) \rangle).
\]
In other words, $H$ is almost compactly generated.
\end{proof}

\begin{proof}[Proof of Theorem~\ref{T2}]
We start with a special case.  Assume there is a dense subgroup $D \subseteq G$ and a continuous, injective homomorphism $\psi : G \rightarrow G$ such that $D = \psi^{-1} (\langle K \rangle)$, for some compact set $K \subseteq G$.  We will show that every $K_\sigma$ subgroup of $G$ is a continuous homomorphic preimage of some compactly generated subgroup of $G^\omega$.  We let $e$ denote the identity element of $G$.

To this end, fix a $K_\sigma$ subgroup $H = \bigcup_n H_n \subseteq G$, with each $H_n$ compact.  For each $n$, let $F_n \subseteq D$ be a finite $\frac{1}{n}$-net for $H_n$.  Let
\[
\tilde H_n = \bigcup_{y \in F_n} y^{-1} \cdot {\rm cl} (B_{\frac{1}{n}} (y) \cap H_n)
\]
and define $C_n \subseteq G^\omega$ to be the set of $\xi \in G^\omega$ such that
\begin{enumerate}
\item $(\forall i < n) (\xi(i) \in \psi (\tilde H_n))$ and
\item $(\forall i \geq n) (\xi (i) \in \psi(H_n))$.
\end{enumerate}
As a product of compact sets, each $C_n$ is a compact subset of $G^\omega$.  In fact, $\bigcup_n C_n$ is itself compact.  To see this, suppose that $\xi_0 , \xi_1 , \ldots \in \bigcup_n C_n$.  In the first place, suppose that there is an $m$ such that infinitely many $\xi_n \in C_m$.  In this case, there is a convergent subsequence of $(\xi_n)$, since each $C_m$ is compact.  On the other hand, suppose that, for each $m$, there are only finitely many $n$ such that $\xi_n \in C_m$.  Extract a subsequence $n_0 < n_1 < \ldots$ such that $\xi_{n_p} \in C_p$, for each $p \in \omega$.  For each $p$, let $\eta_p \in G^\omega$ be such that
\[
\xi_{n_p} (i) = \psi (\eta_p (i)),
\]
for each $i \in \omega$.  By the injectivity of $\psi$, we must have 
\[
\eta_p (i) \in \tilde H_{n_p} \subseteq B_{\frac{1}{n_p}} (e),
\]
for each $p$ and $i < p$.  It follows that $\eta_p \rightarrow \overline e$, as $p \rightarrow \infty$.  Hence, $\xi_{n_p} \rightarrow \overline e$, since $\psi$ is a continuous homomorphism.  

Now define a compact set $K^* \subseteq G^\omega$ by
\[
K^* = \{ \xi \in K^\omega : (\exists^{\leq 1} i) (\xi (i) \neq e)\}
\]
and let $\varphi : G \rightarrow G^\omega$ be given by $\varphi (x) = \overline{\psi(x)}$.  We check that 
\[
H = \varphi^{-1} (\langle K^* \cup \bigcup_n C_n\rangle ).
\]

In the first place, suppose that $x \in H$, say $x \in H_n$.  Let $y \in F_n$ and $z \in \tilde H_n$ be such that $y^{-1} x = z$.  It follows that 
\begin{align*}
\varphi(x)
&= \overline{\psi(x)}\\
&= \underbrace{\left( (\psi(y))^n \concat \overline e\right)}_{\in \langle K^* \rangle} \cdot \underbrace{\left( (\psi(z))^n \concat \overline{\psi(x)} \right).}_{\in C_n}\\
\end{align*}
Hence, $\varphi(x) \in \langle K^* \cup \bigcup_n C_n\rangle$.

On the other hand, suppose that $\varphi (x) \in \langle K^* \cup \bigcup_n C_n\rangle$.  Let $n_0 \in \omega$ and $w$ be a group word, with $\xi_0 , \ldots \xi_k \in C_{n_0}$ and $\eta_0 \ldots , \eta_p \in K^*$ such that 
\[
\varphi (x) = w(\xi_0 , \ldots , \xi_k , \eta_0 , \ldots , \eta_p).
\]
Let $i \geq n_0$ be large enough that 
\[
\eta_0 (i) = \ldots = \eta_p (i) = e.
\]
It follows that 
\begin{align*}
\psi (x) 
&= \varphi (x) (i)\\
&= w (\xi_0 (i) , \ldots , \xi_k (i) , \eta_0 (i) , \ldots , \eta_p (i))\\
&= w (\xi_0 (i) , \ldots , \xi_k (i) , e , \ldots , e) \in \langle \psi (H_{n_0}) \rangle.
\end{align*}
Since $\psi$ is injective, it follows that $x \in \langle H_{n_0} \rangle$.

We now turn to the general case in which there is a dense subgroup $D \subseteq G$, a compact set $K \subseteq G^\omega$, and an injective homomorphism $\psi : G \rightarrow g^\omega$ such that $H = \psi^{-1} (\langle K \rangle)$.

Consider the subgroup
\[
D^* = \{ \xi \in D^\omega : (\forall^\infty i) (\xi (i) = e)\}
\]
and observe that $D^*$ is dense in $G^\omega$.  Define $\psi* : G^\omega \rightarrow (G^\omega)^\omega$ by 
\[
\psi^* (\xi) (n,i) = \psi (x) (i).
\]
Now let
\[
K^* = \{ \eta \in K^\omega : (\exists^{\leq 1} n) (\eta (n) \neq \overline e)\}.
\]
Note that $K^*$ is compact and $D^* = (\psi^*)^{-1} (\langle K^* \rangle)$.  We may therefore apply the special case above to the group $G^\omega$ and use the fact that $G^\omega \cong (G^\omega)^\omega$ to complete the proof.
\end{proof}

\bibliographystyle{plain}

\end{document}